\documentclass{amsart}
\usepackage{amsmath, amsthm, amssymb, amscd, mathrsfs}
\usepackage{graphicx}
\usepackage{pinlabel}

\usepackage{hyperref}

\begin{document}

\newtheorem{theorem}{Theorem}[section]
\newtheorem{lemma}[theorem]{Lemma}
\newtheorem{corollary}[theorem]{Corollary}
\newtheorem{conjecture}[theorem]{Conjecture}
\newtheorem{proposition}[theorem]{Proposition}
\newtheorem{problem}[theorem]{Problem}
\newtheorem*{main_thm}{Main Theorem~\ref{thm:main_theorem}}
\newtheorem*{claim}{Claim}
\newtheorem*{criterion}{Criterion}
\theoremstyle{definition}
\newtheorem{question}[theorem]{Question}
\newtheorem{definition}[theorem]{Definition}
\newtheorem{construction}[theorem]{Construction}
\newtheorem{notation}[theorem]{Notation}
\newtheorem{convention}[theorem]{Convention}
\newtheorem*{warning}{Warning}

\theoremstyle{remark}
\newtheorem{remark}[theorem]{Remark}
\newtheorem{example}[theorem]{Example}
\newtheorem*{case}{Case}

\def\Z{{\mathbb Z}}
\def\N{{\mathbb N}}
\def\R{{\mathbb R}}
\def\Q{{\mathbb Q}}
\def\C{{\mathbb C}}
\def\D{{\mathbb D}}
\def\DD{{\mathcal D}}
\def\H{{\mathbb H}}
\def\L{{\mathcal L}}
\def\M{{\mathcal M}}
\def\T{{\mathcal T}}
\def\E{{\mathcal E}}
\def\X{{\mathcal X}}
\def\SetA{{\mathcal M}}
\def\SetB{{\mathcal M_0}}
\def\SetC{{\mathcal M_1}}
\def\SetAA{{\overline{\mathcal M}}}
\def\SetBB{{\overline{\mathcal M}_0}}
\def\SetCC{{\overline{\mathcal M}_1}}
\def\fz{f_z}
\def\gz{g_z}
\def\Lz{\Lambda_z}
\def\Gz{G_z}
\def\P{{\mathcal P}}

\def\tdLz{L_z}
\def\tdL{L}
\def\diam{\textnormal{diam}}

\newcommand\numberthis{\addtocounter{equation}{1}\tag{\theequation}}
\newcommand{\marginal}[1]{\marginpar{\tiny #1}}

\title{Extreme points in limit sets}

\author{Danny Calegari}
\address{Department of Mathematics \\ University of Chicago \\
Chicago, IL, 60637}
\email{dannyc@math.uchicago.edu}

\author{Alden Walker}
\address{Center for Communications Research \\ San Diego, CA 92121}
\email{akwalke@ccrwest.org}

\date{\today}

\begin{abstract}
Given an iterated function system of affine dilations with
fixed points the vertices of
a regular polygon, we characterize which points in the limit
set lie on the boundary of its convex hull.
\end{abstract}

\maketitle

\setcounter{tocdepth}{1}
\tableofcontents

\section{Introduction}
\label{sec:intro}

\subsection{Background}

Fix $n \ge 2$ and $c \in \C$ with $|c| < 1$.  Let $F_{n,c}$ be the
iterated function system generated by $\{ f_j \}_{j=0}^{n-1}$ with
$f_j : \C \to \C$ defined
\[
f_j(z) = cz + \xi^j,
\]
where $\xi = e^{2\pi i / n}$ is the standard primitive $n$th root of unity.
Let $\Xi_n = \{\xi^j\}_{j=0}^{n-1}$ be the set of all $n$th roots of unity.
The iterated function system $F_{n,c}$, and thus its limit set
$\Lambda_{n,c}$, has rotational symmetry of order $n$ around $0$.
Indeed, it is simple to check that conjugation by the
multiplication-by-$\xi$ map takes $f_j$ to $f_{j+1 \bmod n}$.
Other definitions of iterated function systems with fixed points on
the vertices of a regular $n$-gon are conjugate to $F_{n,c}$;
we have chosen this one to simplify our arguments.
We will only be interested in $n\ge 2$.  This construction is a
simple generalization of the one for $n=2$ initially studied
by~\cite{Barnsley_Harrington}.

Our objective in this note is to characterize the \emph{extreme} points in
$\Lambda_{n,c}$; that is, those which lie on the boundary of its convex hull.
As a consequence, we provide an updated, more thorough version 
of Lemma~7.2.3 in~\cite{Calegari_Koch_Walker}, and we
re-prove~\cite{Himeki}, Proposition~2.1. The material here is
an extracted, expanded piece of~\cite{circle}.

\subsection{Acknowledgements}

Danny Calegari was supported by NSF grant DMS 1405466. 
Alden Walker was partially supported by NSF grant DMS 1203888.

\section{The limit set}
\label{sec:limit_set}

For convenience, we will use $F_{n,c}$ to denote the set of all finite
words in the symbols $\{0,\ldots, n-1\}$, which will be a notational
convenience obviously in bijection with the set of finite
words in the generators $f_j$.  Note that the set of finite words
in the $f_j$ is slightly different from the set of
finite compositions of the $f_j$ because different words may produce
the same function.  We define a map
$\pi:F_{n,c}\times \C \to \C$ by 
\[
\pi( (j_0, j_1, \cdots j_m), z) = (f_{j_0} \circ \cdots \circ f_{j_m})(z).
\]
Let $F_{n,c}^\infty$ denote the set of infinite words
in $\{0,\ldots, n-1\}$.
Given an infinite word $f = (j_0, j_1, \cdots) \in F_{n,c}^\infty$,
the limit $\lim_{m\to\infty} (f_{j_0} \circ \cdots \circ f_{j_m})(z)$ does
not depend on $z$ because $|c| < 1$.  Thus we can extend $\pi$ to
a map $\pi:F_{n,c}^\infty \to \C$.

\begin{lemma}\label{lem:param}
$\pi(F_{n,c}^\infty) = \Lambda_{n,c}$.
\end{lemma}
\begin{proof}
The set $F_{n,c}^\infty$ is compact in the standard Cantor topology, 
and the map $\pi$ is continuous for this topology.  The set of
infinite words is evidently invariant in the sense of
an iterated function system, so the image $\pi(F_{n,c}^\infty)$
is a compact set invariant under $F_{n,c}$, so it is $\Lambda_{n,c}$.
\end{proof}

This parameterization of $\Lambda_{n,c}$ as infinite words in
$F_{n,c}^\infty$ will be our main tool.  Let us see what the
elements of $\Lambda_{n,c}$ can be.  Suppose we have
an infinite word $f = (j_0, j_1, \ldots)$. 
We can compute the associated
point $\pi(f) \in \Lambda_{n,c}$ as
\[
\pi(f) = f_{j_0}(\pi(j_1,\ldots)) = c\pi(j_1,\ldots) + \xi^{j_0}.
\]
By induction, we see that
\[
\pi(f) = \sum_{k=0}^\infty c^k \xi^{j_k}.
\]
That is, the exponents of the $\xi$ are exactly
the letters in the infinite word $f$.
There is a clear geometric picture for such sums: at every step $k$, 
we select which power $j_k$ of $\xi$ we would like to multiply
by $c^k$ and accumulate into the sum.
Note that the set $\{c^k\xi^j\}_{j=0}^{n-1}$ is a rotated, scaled
copy of the roots of unity.  We imagine a constellation of potential vectors
at each step, and we can choose to add in any one of them.  See
Figure~\ref{fig:constellation}.

\begin{figure}[ht]
\begin{center}
\labellist
\pinlabel $c$ at 85 90
\pinlabel $\xi^0$ at 123 80
\pinlabel $\xi^1$ at 80 145
\pinlabel $\xi^2$ at 7 119
\pinlabel $\xi^3$ at 3 40
\pinlabel $\xi^4$ at 80 20
\endlabellist
\includegraphics[width=0.9\textwidth]{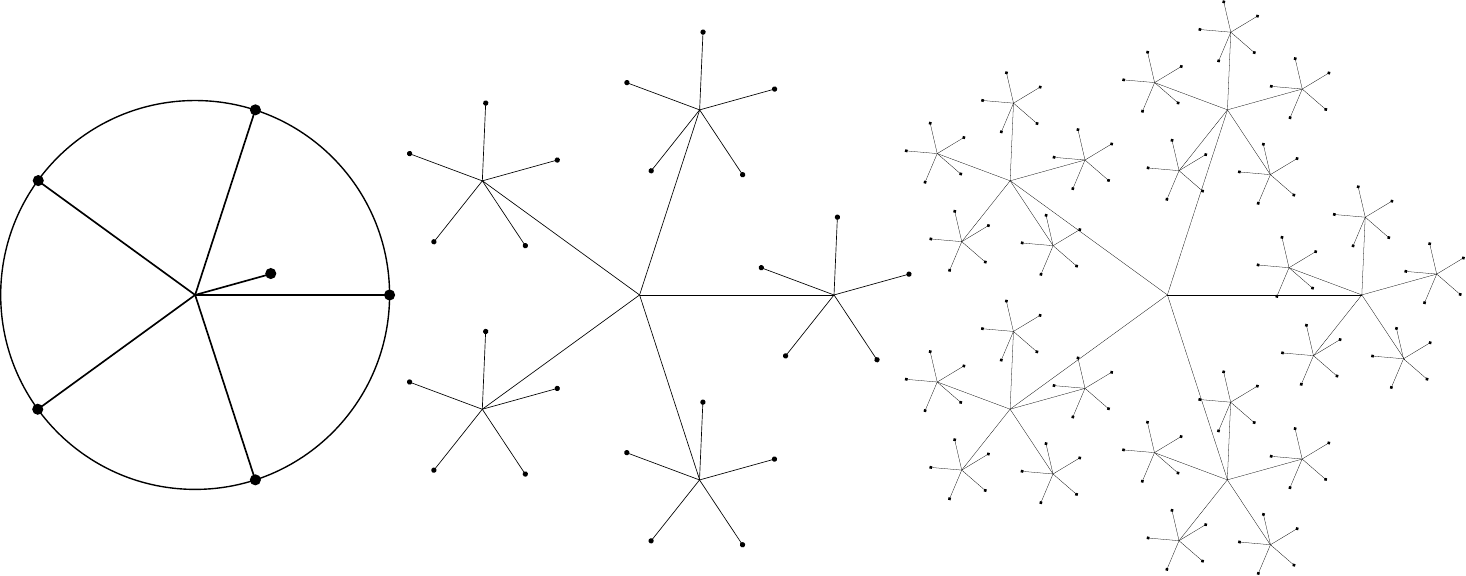}
\end{center}
\caption{Roots of unity for $n=5$ and $c$ as shown, left, and
partial sums at the first and second steps of computing $\pi(f)$ for
all possible words.  The points in $\Lambda_{n,c}$ are the result of
repeating this procedure infinitely many times.}
\label{fig:constellation}
\end{figure}
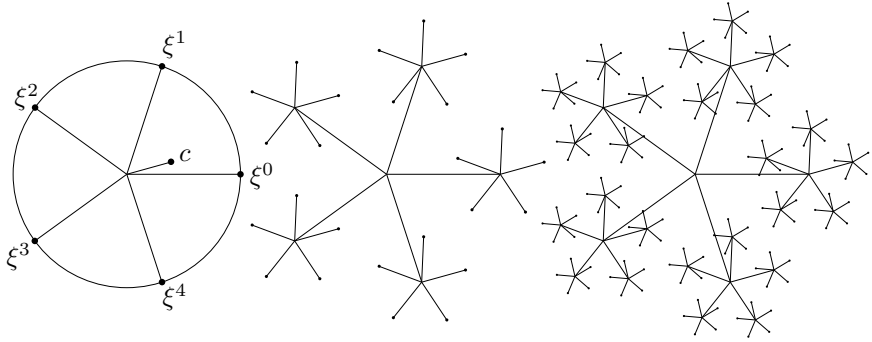

\section{Angles and hyperplanes}

Any point in $\C$ can be thought of as a vector in $\R^2$, and 
any vector $v \in \R^2$ induces a (real) linear function on $\C$
given by taking the standard inner product with $v$.  We will be
interested in the linear functions coming from the unit circle, 
so we denote by $v_\theta:\C \to \R$ the linear function given by
taking the real inner product with the point $e^{2\pi i\theta}$.

\begin{remark}
For convenience, we will normalize all angles to lie
in $[0,1]$.  Thus, as above, we refer to the argument
of $e^{2\pi i \theta}$ as ``the angle $\theta$''.
\end{remark}

The level sets of $v_\theta$ are (oriented) affine hyperplanes, and
we say that they lie at angle $\theta$.  Note that it is
the normal vector that points in the direction $\theta$.
Such a hyperplane $H$ \emph{supports}
a closed set $X \subseteq C$ if $H$ intersects $X$ and $X$
is contained in the closure of one of
the two complements of $H$ in $\C$.  Equivalently, $H$ supports $X$ if 
the value of $v_\theta$ on $X\cap H$ (which must be constant) is maximal
over all points in $X$.  The points which lie in the intersection of $X$
and a hyperplane at angle $\theta$ which supports it are
\emph{extreme points} at angle $\theta$.

Turning to our problem of interest,
we now consider the extreme points in $\Lambda_{n,c}$
in terms of the parameterization of $\Lambda_{n,c}$
from Section~\ref{sec:limit_set} and Lemma~\ref{lem:param}.
The extreme points at angle $\theta$
are those points $p$ such that $v_\theta(p)$ is maximal over
$\Lambda_{n,c}$.  Denote the set of extreme points in $\Lambda_{n,c}$ at angle $\theta$
by $E_{n,c,\theta}$.
Let us be given an infinite word $f = (j_0,\ldots) \in F_{n,c}^\infty$.
We say that $f$ is an \emph{extreme word} at angle $\theta$
if $\pi(f)$ is an extreme point at angle $\theta$.  Denote the set of
extreme words in $F_{n,c}$ at angle $\theta$ by $W_{n,c,\theta}$.
It is cleaner to characterize the extreme words $W_{n,c,\theta}$ in
$F_{n,c}^\infty$ and then translate that understanding to
the set of extreme points $E_{n,c,\theta}$ in $\Lambda_{n,c}$.

We have
\[
\pi(f) = \sum_{k=0}^\infty c^k \xi^{j_k},
\]
so 
\[
v_\theta(\pi(f)) = \sum_{k=0}^\infty v_\theta(c^k\xi^{j_k}).
\]
Now $v_\theta$ is real linear but not complex linear,
so we cannot pull the coefficient $c^k$ out.  It is difficult
to understand exactly what this sum is equal to, but it is
not difficult to maximize it: we just need to maximize each summand.
For each $k$, the value of $v_\theta(c^k\xi^{j_k})$ is maximized
when the vector $c^k\xi^{j_k}$ is the closest to angle $\theta$
among the vectors $\{c^k\xi^j\}_{j=0}^{n-1}$.  So to construct 
an extreme point $p$ at angle $\theta$, we choose,
for each $k$, the letter $j_k$ such that $c^k\xi^{j_k}$ is closest to $\theta$.
It is possible that there will be multiple choices which are equidistant from
$\theta$, in which case there are exactly two options.
This produces a very clean geometric picture of each extreme
point, which is illustrated in Figure~\ref{fig:extreme}.

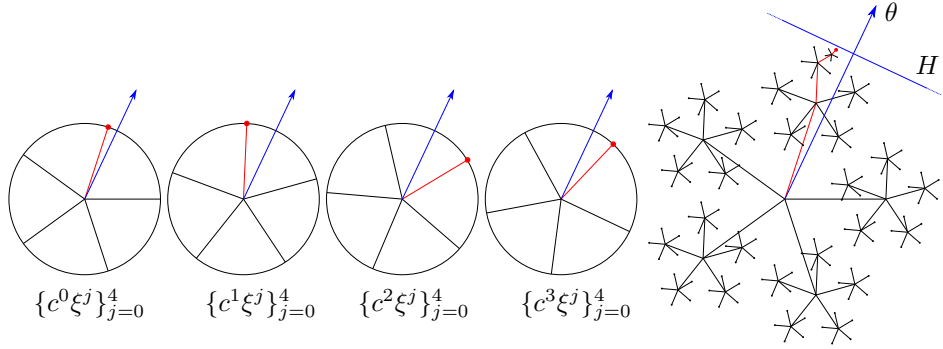
\begin{figure}[ht]
\begin{center}
\labellist
\pinlabel $\{c^0\xi^j\}_{j=0}^4$ at 45 22
\pinlabel $\{c^1\xi^j\}_{j=0}^4$ at 140 22
\pinlabel $\{c^2\xi^j\}_{j=0}^4$ at 225 22
\pinlabel $\{c^3\xi^j\}_{j=0}^4$ at 315 22
\pinlabel $\theta$ at 490 185
\pinlabel $H$ at 510 155
\endlabellist
\includegraphics[width=0.98\textwidth]{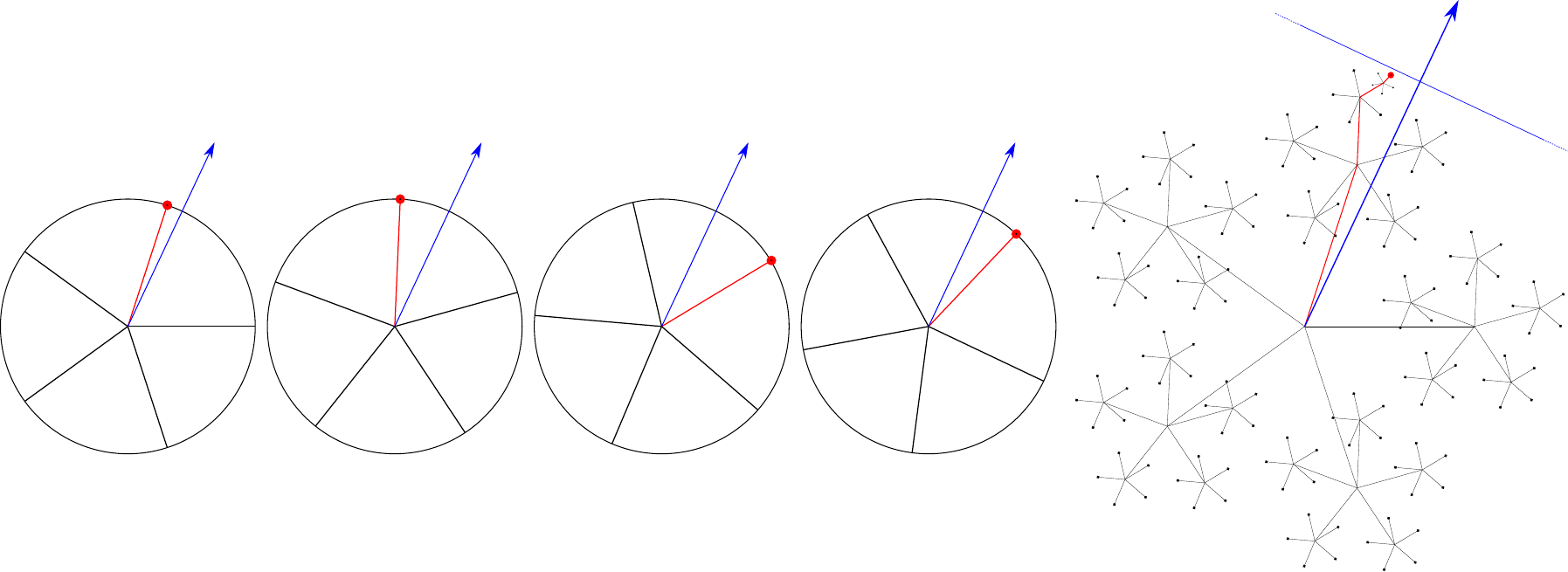}
\end{center}
\caption{To produce an extreme point at angle $\theta$, we
consider the linear function $v_\theta$, where $e^{2\pi i\theta}$
is the arrow indicated in blue.  For each $k$, we are forced
to choose $j_k$ such that the $v_\theta(c^k\xi^{j_k})$ 
is maximal for the vectors $\{c^k\xi^j\}_{j=0}^{n-1}$.  The partial
sum after $4$ steps is shown on right.}
\label{fig:extreme}
\end{figure}

Everything which follows is basically a direct observation
from Figure~\ref{fig:extreme}.  As mentioned above, for each
given $k$, there are two possibilities for the set of
function values $\{v_\theta(c^k\xi^j)\}_{j=0}^{n-1}$.  Either
all of these values are distinct, in which case the
choice of coordinate $j_k$ is forced, or there is some $j_k$ such that
$v_\theta(c^k\xi^{j_k}) = v_\theta(c^k\xi^{j_k+1})$.
This happens exactly when the angle $\theta$ is equidistant from
the arguments of $c^k\xi^{j_k}$ and $c^k\xi^{j_k+1}$.
Here and in what follows, we take all powers of $\xi$ modulo $n$.
The two situations are shown in Figure~\ref{fig:extreme_choices}.

\begin{figure}[ht]
\begin{center}
\includegraphics[width=0.50\textwidth]{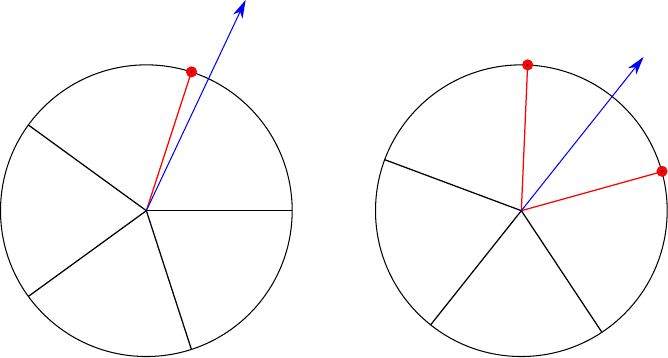}
\end{center}
\caption{When $\theta$ is equidistant from the arguments
of $c^k\xi^{j_k}$ and $c^k\xi^{j_k+1}$, as on the right, both $j_k$ and
$j_k+1$ are allowed as coordinate $k$ in an extreme
point at angle $\theta$.}
\label{fig:extreme_choices}
\end{figure}

We now formalize the observations from Figures~\ref{fig:extreme}
and~\ref{fig:extreme_choices}.
Let $n \in \mathbb{N}$, $c = r_ce^{2\pi i\phi}$ and $\theta \in [0,1]$ be given.
For each $k\ge 0$, define a set $J_{n,c,\theta,k} \subset \{0,\ldots, n-1\}$ as follows.
If there exists an integer $m \in [0,n-1]$ such that
\[
k\phi + \frac{m}{n} + \frac{1}{2n} \equiv \theta \bmod 1,
\]
then $J_{n,c,\theta,k} = \{m,m+1\}$ (recall all indices are taken modulo $n$).
Otherwise, $J_{n,c,\theta,k} = \{m\}$ where $k\phi + m/n$ is closest to $\theta$.

\begin{lemma}\label{lem:extreme_words}
In the above notation,
\[
W_{n,c,\theta} = \prod_{k=0}^\infty J_{n,c,\theta,k}.
\]
\end{lemma}

That is, the set of extreme words in $F_{n,c}^\infty$
at angle $\theta$ is the infinite cartesian product of the sets $J_{n,c,\theta,k}$.
Or, more simply, to enumerate all the extreme words, we can go one letter
at a time.  When we need to decide on letter $k$, we consult
$J_{n,c,\theta,k}$ to 
see what the allowed letters are, and any choice is allowed (all constraints
on the letters are local).

\begin{proof}[Proof of Lemma~\ref{lem:extreme_words}]
The proof is contained in the preceding discussion: if we are given
an extreme word $f= (j_0,\ldots)$, and we write the linear function
distributed over the sum:
\[
v_\theta(\pi(f)) = \sum_{k=0}^\infty v_\theta(c^k\xi^{j_k}),
\]
then is suffices to optimize each term independently.  The definition
of $J_{n,c,\theta,k}$ is the condition to check that
there are two equidistant options
as in Figure~\ref{fig:extreme_choices}.
\end{proof}

\section{Extreme point alternatives}

\subsection{Main theorem}

We now follow the line of reasoning begun in Lemma~\ref{lem:extreme_words}.
 We first state
our main theorem and then spend the rest of the section proving 
and explaining it.

\begin{theorem}\label{thm:alternatives}
As above, let $c = r_ce^{2\pi i\phi}$.
\begin{enumerate}
\item \label{thm_part:irrational} If $\phi \notin \Q$, then
\begin{enumerate}
\item For all $\theta$, we have $|W_{n,c,\theta}| = |E_{n,c,\theta}| \in \{1,2\}$.
\item The set of $\theta$ such that $|W_{n,c,\theta}|=1$ is dense in $[0,1]$.
\item The set of $\theta$ such that $|W_{n,c,\theta}|=2$ is dense in $[0,1]$.
\item $\Lambda_{n,c}$ is not convex.
\end{enumerate}

\item \label{thm_part:rational} If $\phi\in\Q$ with $\phi = p/q$ in lowest terms, then let
$b$ be such that $bn = \textnormal{lcm}(n,q)$ (e.g.
if $\gcd(n,q)=1$, then $b=q$).  Let
\[
\Theta = \left\{\ell\phi + \frac{m}{n} + \frac{1}{2n}\right\}_{\ell=0,m=0}^{b-1,n-1}.
\]
Then
\begin{enumerate}
\item For $\theta \in \Theta$, we have that $E_{n,c,\theta}$
is itself the limit set of an iterated function system
with two generators conjugate to $F_{2,|c|^b}$.
That is, a Cantor set with dilation factor $|c|^b$
(or an interval if $|c|^b \ge 1/2$).
\item For $\theta \notin \Theta$, we have
$|W_{n,c,\theta}|=|E_{n,c,\theta}|=1$.
\item The convex hull of $\Lambda_{n,c}$ is a polygon
with $nb$ sides at the angles in $\Theta$.
\item \label{thm_part:convex} If $\Lambda_{n,c}$ is convex, then $|c| \ge 2^{-1/b}$.
\end{enumerate}
\end{enumerate}
\end{theorem}

\subsection{The irrational case}

Before the formal proof of the rational case of Theorem~\ref{thm:alternatives}, we give the picture of the
proof, which is
actually much more convincing.  If we construct an infinite
word $f \in F_{n,c}^\infty$ by building each summand in the infinite sum
$\pi(f)$ as in Figure~\ref{fig:constellation}, then for each $k$ we have $n$
options of which power $j_k$ to use in the term $c^k\xi^{j_k}$ to accumulate.
As we have seen in Figures~\ref{fig:extreme} and~\ref{fig:extreme_choices},
if we are building an extreme word for angle $\theta$, our choice is dictated
by which $j_k$ makes $c^k\xi^{j_k}$ as close to $\theta$ as possible.
Let us consider what the set of options $\{c^k\xi^j\}_{j=0}^{n-1}$ looks 
like when $\phi \notin \Q$.  The set $\Z\phi$ is dense in the interval $[0,1]$,
and there are no distinct integers $k,k'$ such that $k\phi\equiv k'\phi \bmod 1$.
That is, the constellation of choices never repeats.  Thus, if we ever happen
to stumble upon an index $k$ with $|J_{n,c,\theta,k}|=2$, it will never happen again.
Figure~\ref{fig:irrational_constellations} gives a picture.

\begin{figure}[ht]
\begin{center}
\labellist
\pinlabel $c$ at 67 51
\pinlabel $\theta$ at 75 100
\pinlabel $\theta$ at 167 100
\endlabellist
\includegraphics[width=0.50\textwidth]{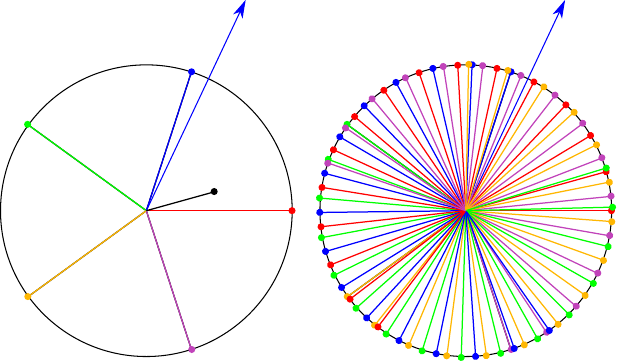}
\end{center}
\caption{When $\phi \notin \Q$, the constellations $\{c^k\xi^j\}_{j=0}^{n-1}$
are dense (and not periodic).  Thus, if we ever find a $k$
as in Figure~\ref{fig:extreme_choices} with two equidistant choices for some given $\theta$,
that is the only $k$ for which it occurs.  The figure shows the constellations
for $k=0,\ldots,15$.}
\label{fig:irrational_constellations}
\end{figure}

\begin{proof}[Proof of Theorem~\ref{thm:alternatives}(\ref{thm_part:irrational})]
To prove (a), it suffices to show that for all $\theta$,
there is at most one $k$ such that $|J_{n,c,\theta,k}| = 2$.
Towards a contradiction, assume there two distinct such values $k,k'$.
Then there exist $m,m'$ integers with
\[
k\phi + \frac{m}{n} + \frac{1}{2n} \equiv \theta \equiv k'\phi + \frac{m'}{n} + \frac{1}{2n} \bmod 1
\]
So $(k-k')\phi \equiv (m-m')/n \bmod 1$.  Since $k\ne k'$, this
implies that $\phi$ is rational, which is a contradiction.

To prove (b) and (c), we exhibit sets of $\theta$ with the desired
properties.  First, for any integer $\ell$, set $\theta = \ell\phi$.
Then suppose there is any integer $k$ with
\[
k\phi + \frac{m}{n} + \frac{1}{2n} \equiv \theta =\ell\phi \bmod 1.
\]
We conclude that $(k-\ell)\phi \equiv m/n + 1/2n \bmod 1$.  The expression
on the right is never $0$, which implies that $\phi \in \Q$,
a contradiction.  Thus for all $k$ we have $|J_{n,c,\theta,k}|=1$, meaning
$|W_{n,c,\theta}|=1$ and hence $|E_{n,c,\theta}=1$.

Next, we set $\theta = \ell\phi + 1/2n$ for any integer $\ell$.  Now we have
\[
k\phi + \frac{m}{n} + \frac{1}{2n} \equiv \ell\phi + \frac{1}{2n} \bmod 1,
\]
so $(k-\ell)\phi \equiv m/n \bmod 1$.  This does have exactly one solution, where $k=\ell$ and $m=0$.  Hence there is exactly one $k$ such that
$|J_{n,c,\theta,k}|=2$, so $|W_{n,c,\theta,k}|=2$.  In general, it is
difficult to know the size of $E_{n,c,\theta}$ from knowing
$W_{n,c,\theta}$.  However, in this case the two extreme words differ
in exactly one letter, so their images under $\pi$ differ in exactly one
(nonzero) summand, so there are exactly two extreme points at
angle $\theta$.

We have proved that for all $\theta \in \phi\Z$, we have
$|W_{n,c,\theta}|=|E_{n,c,\theta}|=1$ and for all
$\theta \in \phi\Z + 1/2n$, we have $|W_{n,c,\theta}|=|E_{n,c,\theta}|=1$.
These sets are dense in the circle (represented here by the interval
$[0,1]$) because $\phi \in \Q$, and we have proved (b) and (c).

Claim (d) is an immediate consequence of (c): if we exhibit any 
single $\theta$ with $|E_{n,c,\theta}|=2$, then every point on the
line segment between these two points lies in the convex hull of
$\Lambda_{n,c}$ but not in $\Lambda_{n,c}$ itself, meaning that
$\Lambda_{n,c}$ is not convex.
\end{proof}

\subsection{The rational case}

We now turn to understanding the extreme points in $\Lambda_{n,c}$
when $\phi \in \Q$.  This case is simpler in some ways
but more technical and interesting in others.  In particular,
we will now find sets of extreme words at certain angles which are infinite.
In order to simplify the reasoning, we will first
restrict $c = r_ce^{2\pi i\phi}$ such that $\phi\in\Q$
behaves nicely with respect to $n$, meaning that $\phi=p/q$ with $\gcd(n,q)=1$.
The general case will follow by proving a technical lemma
(Lemma~\ref{lem:period}) and observing that with this lemma in hand,
the general argument
is essentially the same. 
The truely helpful picture of the situation,
analogous to Figure~\ref{fig:irrational_constellations}, is shown
in Figure~\ref{fig:rational_constellations}.  When $\phi\in\Q$, one can see that
the constellations $\{c^k\xi^j\}_{j=0}^{n-1}$ are periodic in $k$.
The restriction that $\gcd(n,q)=1$ ensures that this period is
easy to compute (it is $nq$).

\begin{figure}[ht]
\begin{center}
\labellist
\pinlabel $c$ at 41 69
\pinlabel $\theta$ at 75 100
\pinlabel $\theta$ at 167 100
\endlabellist
\includegraphics[width=0.50\textwidth]{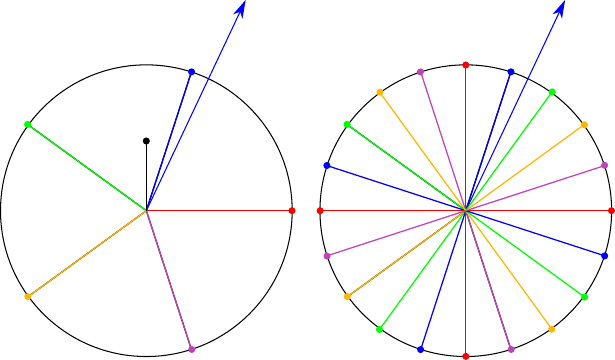}
\includegraphics[width=0.4\textwidth]{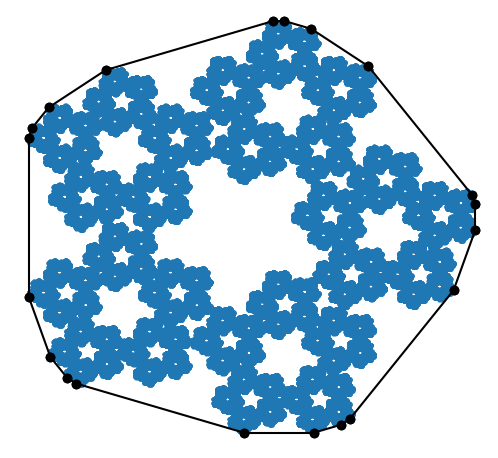}
\end{center}
\caption{When $\phi \in \Q$, the constellations $\{c^k\xi^j\}_{j=0}^{n-1}$
are periodic.  Here $\phi = 1/4$ and $n=5$.  The constellations for
all $k$ in the same residue class modulo $4$ are the same.
The limit set $\Lambda_{5,(2/5)i}$ is shown on the right
with the supporting hyperplanes at the angles in $\Theta$.  The
Cantor set of extreme points in each face appears to be just two points
because the dilation factor $|c|^q \approx 0.01$ is so small.}
\label{fig:rational_constellations}
\end{figure}

\begin{proof}[Proof of Theorem~\ref{thm:alternatives}(\ref{thm_part:rational}) when $\gcd(n,q)=1$]
Because $\phi=p/q$, we have $c^q = r_c^q$, i.e. the
angle of $c^k\xi^j$ is the same as the angle of $c^{k+q}\xi^j$.
This implies that the angles of the constellations
$\{c^k\xi^j\}_{j=0}^{n-1}$ are periodic in $k$, and
for all $k$ and $\theta$, we have
$J_{n,c,\theta,k} = J_{n,c,\theta,k+q}$.
The fact that $\gcd(q,n)=1$ means that we do not have any additional
equalities --- that is, there are no $k,k',j,j'$ 
with $0\le k,k' < q-1$ and $0\le j,j' < n-1$ with
$\arg(c^k\xi^j) = \arg(c^{k'}\xi^{j'})$.

The combination of these two facts implies that for exactly
the angles listed in $\Theta$, we have $|E_{n,c,\theta}| > 1$.
Now consider a specific $\theta \in \Theta$.  The product
$\prod_{k=0}^{q-1} J_{n,c,\theta,k}$ contains exactly two
words, since there will be exactly one $k$ for which we
have two options.  Call these words $w_0$ and $w_1$.
Now if $f \in W_{n,c,\theta}$ is any extreme word for $\theta$,
the discussion above shows that both words
$w_0 + f$ and $w_1 + f$ (where $+$ means word concatenation) are
also extreme words for $\theta$.  Pushing forward under $\pi$,
this means that the limit set of the iterated function
system generated by the functions
\[
z \mapsto c^qz + w_0 \qquad z \mapsto c^qz + w_1
\]
is exactly the set of extreme points $E_{n,c,\theta}$.
Since $c^q = r_c^q$ is real, this is conjugate to
the function system $F_{2,|c|^q}$, which has limit
set as described.

If we consider any $\theta_0, \theta_1$ consecutive elements
of $\Theta$, then note that because of the discreteness of
the constellations, for any two $\theta,\theta'$ with
$\theta_0 < \theta,\theta' < \theta_1$, we have
$J_{n,c,\theta,k} = J_{n,c,\theta',k}$ for all $k$.
Therefore, all these angles share one extreme point.

The previous two paragraphs imply facts (a), (b), and (c).
To see (d), note that if we are to have $\Lambda_{n,c}$ convex,
then for $\theta\in\Theta$, the set of extreme points $E_{n,c,\theta}$
must be an interval.  Thus $|c|^q \ge 1/2$, or $|c| \ge 2^{-1/q}$.
\end{proof}

For the general case when we do not necessarily have $\gcd(n,q) = 1$,
the only thing we need to determine is the periodicity of
(the angles of)
the constellation $\{c^k\xi^j\}_{j=0}^{n-1}$ as $k$ varies.
For simplicity, denote by $C_k$ the set of angles
$\{\arg(c^k\xi^j)\}_{j=0}^{n-1}$.  We are careful to say a \emph{set},
because while assuming, as above, that $\phi = p/q$ in lowest terms,
we certainly have that $C_k = C_{k+q}$, but we
might have $C_k = C_{k+b}$ for some $b<q$ where the elements are
not in the same order.  As an example, if
$\phi = 1/6$ and $n=3$, then $C_k = C_{k+2}$.

\begin{lemma}\label{lem:period}
In the above notation, let $b$ be the smallest positive integer
such that $b(1/q) = a(1/n)$ for some integer $a$.
Equivalently, let $bn = \textnormal{lcm}(n,q)$. Then
for all $k$, we have $C_k = C_{k+b}$ and $C_k \cap C_{k+i} = \varnothing$
for $0 < i < b$.
\end{lemma}
\begin{proof}
It is immediate that $C_k = C_{k+b}$ when $b(1/q) = a(1/n)$
because
\[
\arg(c^{k+b}\xi^j) = \phi(k+b) + j/n = \phi k + (p/q)b + j/n = \phi k + (pa+j)/n = \arg(c^k\xi^{pa+j})
\]
If we supposed towards a contradiction that for $i < b$ we have
$C_k \cap C_{k+i} \ne \varnothing$, then an analogous chain of equalities
gives that there are $j_1,j_2$ with
\[
(p/q)(k+i) + j_1/n = (p/q)k + j_2/n,
\]
so $(p/q)i = (j_1-j_2)n$.  But $b$ is the smallest positive integer
such that this can hold, which contradicts that $i<b$.
\end{proof}

\begin{proof}[Proof of Theorem~\ref{thm:alternatives}(\ref{thm_part:rational}) in the general case]
This is essentially a corollary of the above Lemma~\ref{lem:period}:
the proof of Theorem~\ref{thm:alternatives}(\ref{thm_part:rational}) in
the case when $\gcd(n,q)=1$ depends only on
characterizing the periodicity of the angles $C_k$,
which is done by Lemma~\ref{lem:period}.
\end{proof}

\subsection{Final remarks}

A consequence of Theorem~\ref{thm:alternatives}(\ref{thm_part:convex})
is that for any $n$ and $c$, if $\Lambda_{n,c}$ is convex,
then $|c| \ge 1/2$ (see~\cite{Himeki}, Proposition~2.1).
Furthermore, this bound is sharp from the perspective of the
set of extreme points:
there are $c$ (in particular, $c=1/2 + 0i$) with $|c|=1/2$
such that every $E_{n,c,\theta}$ is an interval.
Note that this does \emph{not} imply that $\Lambda_{n,c}$ is convex:
consider the Sierpinski triangle,
where the boundary of the limit set is a finite sided polygon, but
the limit set itself is complicated.  Figure~\ref{fig:examples}
shows some more examples of limit sets and their supporting hyperplanes.

\begin{figure}[ht]
\begin{center}
\includegraphics[width=0.43\textwidth]{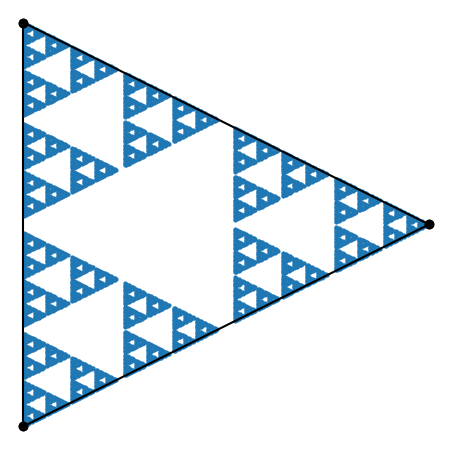}
\includegraphics[width=0.46\textwidth]{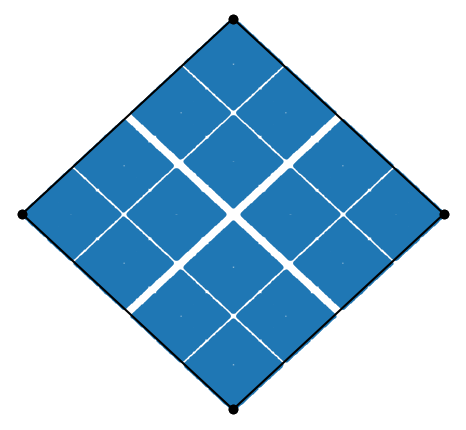}

\includegraphics[width=0.46\textwidth]{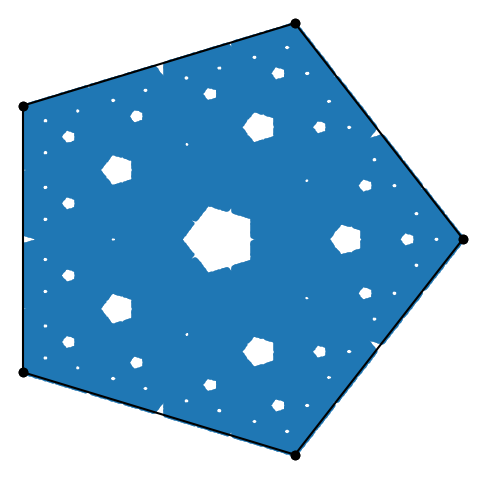}
\includegraphics[width=0.46\textwidth]{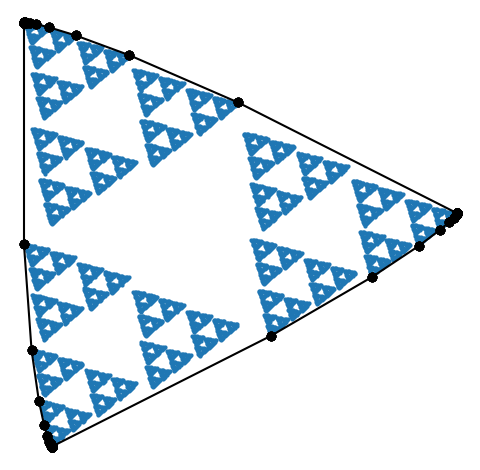}
\end{center}
\caption{A variety of limit sets with $c = 0.48 + 0i$ and supporting hyperplanes,
showing how the limit set may or may not be convex, even when
the sets of extreme points are intervals (the dilation
amounts here are just shy of $0.5$ to show the detail).  A small twist
of $\phi = 1/100$ in the lower right produces a much more
interesting collection of extreme points, and the boundary 
of the convex hull is now a polygon with $300$ sides.}
\label{fig:examples}
\end{figure}


\begin{thebibliography}{99}
\bibitem{Barnsley_Harrington}
	M. Barnsley and A. Harrington,
	\emph{A Mandelbrot set for pairs of linear maps},
	Phys. D. {\bf 15} (1985), no. 3, 421--432
\bibitem{circle}
D. Calegari and A. Walker
\emph{Circle actions on the boundary of Schottky space},
in preparation.
\bibitem{Calegari_Koch_Walker}
	D. Calegari, S. Koch and A. Walker,
	\emph{Roots, Schottky semigroups, and a proof of Bandt's Conjecture},
	Ergodic Theory and Dynamical Systems {\bf 37} (2017) no. 8, 2487--2555. 
	doi:10.1017/etds.2016.17	
	
\bibitem{Himeki}
Himeki, Y. and Ishii, Y. \emph{${\mathcal{M}}_{4}$ is regular-closed}, Ergodic Theory and Dynamical Systems, 1-8 doi:10.1017/etds.2018.27
\end{thebibliography}
\end{document}